\documentclass[10pt,reqno]{amsart}
\usepackage{floatflt}
\usepackage{xypic}
\usepackage{rotating}
\usepackage{bbm}
\usepackage{mathrsfs}
\usepackage{diagbox}
\usepackage{dutchcal}
\usepackage{extarrows}
\usepackage{cite}
\usepackage{amsfonts} 
\usepackage[dvipsnames,usenames]{color}
\usepackage{graphicx,latexsym,bm,amsmath,amssymb,verbatim,multicol,lscape}
\makeatletter
\@namedef{subjclassname@2020}{
\textup{2020} Mathematics Subject Classification}
\makeatother
\newtheorem{theorem}{Theorem} [section]

\newtheorem{conjecture}[theorem]{Conjecture}
\newtheorem{lemma}[theorem]{Lemma}

\numberwithin{equation}{section}

\def\bew{\begin{widetext}}
\def\eew{\end{widetext}}
\def\be{\begin{equation}}
\def\ee{\end{equation}}
\def\bea{\begin{eqnarray}}
\def\eea{\end{eqnarray}}

\allowdisplaybreaks
\begin{document}
\title{Gcd-closed sets and divisibility among power LCM matrices}

\begin{abstract}
Let $a,b$, and $n$ be positive integers and let $S=\{x_1, \cdots, x_n\}$
be a set of $n$ distinct positive integers. We denote by $[S^a]$ 
the $n\times n$ matrix having the $a$th power of the 
least common multiple of $x_i$ and $x_j$ as
its $(i,j)$-entry. For $x\in S$, let $G_{S}(x)=\{d\in S: d<x, d|x \
{\rm and} \ (d|y|x, y\in S)\Rightarrow y\in \{d,x\}\}$. In this article,
we show that $[S^a]$ divides $[S^b]$ in the ring of $n\times n$ matrices 
over the integers if $a|b$ and $S$ is gcd closed (i.e., $\gcd(x_i,x_j)\in S$ 
for all integers $i$ and $j$ with $1\le i,j\le n$) and satisfies 
the condition $\mathcal G$ (that is, for any $x\in S$, 
either $G_S(x)$ contains at most one element, or $G_S(x)$
contains at least two elements and satisfies that ${\rm lcm}(y_1,y_2)=x$ and
$\gcd(y_1,y_2)\in G_S(y_1)\cap G_S(y_2)$ for any $\{y_1,y_2\}\subseteq G_S(x)$).
This confirms a conjecture of Hong proposed in (2026, 
Bull. Aust. Math. Soc., 113, 231-243).
\end{abstract}

\author[G.Y. Zhu]{Guangyan Zhu}
\address{School of Mathematics and Statistics,
Hubei Minzu University, Enshi 445000, P.R. China}
\email{2009043@hbmzu.edu.cn}
\thanks{The research was supported in part by the Startup
Research Fund of Hubei Minzu University for Doctoral Scholars
(Grant No. BS25008)}

\maketitle
		
\section{Introduction}
Let $a,b$ and $n$ be positive integers.
For any integers $x$ and $y$, we denote by $(x, y)$ the greatest common
divisor of integers $x$ and $y$ and by $[x,y]$ their least common multiple.
Let $\mathbb Z$ denote the ring of integers and let $|T|$ stand for the
cardinality of a finite set $T$ of integers. Let $f$ be an arithmetic
function and let $S=\{x_1, \cdots, x_n \}$. Let $(f(S))$ and $(f[S])$
denote the $n\times n$ matrices whose $(i,j)$-entries are $f((x_i, x_j))$
and $f([x_i, x_j])$, respectively. Define $\langle n\rangle:=\{1,\cdots, n\}$.
For any real number $e$, let $\xi_e$ be the arithmetic function defined for any
positive integer $x$ by $\xi_e(x)=x^e$. The $n\times n$ matrix $(\xi_a(x_i, x_j))
=((x_i, x_j)^a)$ (abbreviated by $(S^a)$) and $(\xi_a[x_i, x_j])=([x_i, x_j]^a)$
(abbreviated by $[S^a]$) are called {\it $a$th power GCD matrix} on $S$ and
{\it $a$th power LCM matrix} on $S$, respectively. The set $S$ is
said to be {\it factor closed} (written as FC for short) if
$(x\in S$, $d>0$, $d|x)\Rightarrow d\in S$. We say that $S$ is
{\it gcd closed} if $S$ contains $(x_i, x_j)$ for all integers
$i$ and $j$ with $1\le i,j\le n$. Obviously, any FC set contains
1 and is gcd closed but the converse is not true. For instance,
if the set $S$ is FC, then for any integer $x>1$, the set
$xS:=\{xy| y\in S\}$ is gcd closed but not FC since $1\not\in xS$.
One hundred and fifty years ago, Smith \cite{[S-PLMS1875]} proved that
\begin{align}\label{eq1.1}
\det([x_i,x_j])=\prod_{k=1}^n\varphi(x_k)\pi(x_k)
\end{align}
if $S$ is FC,
where $\varphi$ is the Euler's totient function and
$\pi$ is the multiplicative function defined for the
prime power $p^r$ by $\pi(p^r)=-p$. After that,
many generalizations of Smith's determinant
(\ref{eq1.1}) and related results were published
(see, for example, \cite{[A-PJM1972]}-\cite{[LT-DM2018]}
and \cite{[TL-LAA2013]}-\cite{[ZY-PAMS2026]}). In particular,
an elegant result was gotten by Hong, Hu and Lin
\cite{{[HHL-AMH2016]}} stating that for any integer $n\ge 2$,
one has
$$
\det([i,j])_{2\le i,j\le n}=\Big(\prod_{k=1}^n
\varphi (k)\pi(k)\Big)\sum_{t=1\atop t \
\text{is square free}}^n\frac{t\mu(t)}{\varphi (t)},
$$
where $\mu$ is the M$\ddot{\rm o}$bius function and
an integer $x\ge 1$ is called {\it square free} if
$x$ is not divisible by the square of any prime number.

Hong \cite{[H-JA1999]} introduced the concept of greatest-type divisor
when he solved the famous Bourque-Ligh conjecture \cite{[BL-LAA1992]}.
For any integer $x\in S$, if
$(y<x, y|z|x \ {\rm and} \ y,z\in S)\Rightarrow z\in\{y,x\},$
then $y$ is called a {\it greatest-type divisor} of $x$. Let
$G_S(x):=\{y\in S: y \ \text{is a greatest-type divisor of}
\ x \ {\rm in} \ S\}$. Let $M_n({\mathbb Z})$ stand for the ring of
$n\times n$ matrices over the integers. In 1992, Bourque and Ligh \cite{[BL-LAA1992]}
proved that $(S)$ divides $[S]$ in the ring $M_n(\mathbb Z)$ if $S$ is FC.
Namely, there exists a $B\in M_n({\mathbb Z})$ such that $[S]=B(S)$
or $[S]=(S)B$. In 2002, Hong \cite{[H-LAA2002]} showed that such
a factorization is not true when $S$ is gcd closed and $\max_{x \in S}\{|{G_S(x)}|\}=2$.

On the other hand, Hong \cite{[H-LAA2008]} initiated the investigation of
divisibility among power LCM matrices. It was proved in \cite{[H-LAA2008]}
that $[S^a]|[S^b]$ if $a|b$ and $S$ is a divisor chain (that is,
$x_{\sigma(1)}|\cdots|x_{\sigma(n)}$ for a permutation $\sigma$ of
$\langle n\rangle$). Evidently, a divisor chain is gcd closed but
not conversely. Tan and Li \cite{[TL-LAA2013]} generalized Hong's
result \cite{[H-LAA2008]} by proving that $[S^a]|[S^b]$ holds
in the ring $M_{|S|}({\mathbb Z})$ if $a|b$ and $S$ consists
of finitely many coprime divisor chains with $1\in S$,
and that such divisibility relation is not true if
$a\nmid b$. Zhu and Li \cite{[ZL-BAMS2023]} confirmed
the conjecture of Hong raised in \cite{[H-LAA2008]}
stating that if $a|b$ and $S$ is a gcd-closed set with
$\max_{x\in S}\{|G_S(x)|\}=1$, then $[S^a]\mid [S^b]$
in the ring $M_n({\mathbb Z})$. In 2026, Hong \cite{[H-BAMS25]}
established the same divisibility result when $S$ is FC.
As in \cite{[H-BAMS25]}, for any set $S$ of positive integers and for
any $x\in S$ with $|G_S(x)|\ge 2$, we say that the two distinct greatest-type
divisors $y_1$ and $y_2$ of $x$ in $S$ {\it satisfy the condition
$\mathcal{G}$} if $[y_1,y_2]=x$ and $(y_1,y_2)\in G_S(y_1)\cap G_S(y_2)$.
We say that $x$ {\it satisfies the condition $\mathcal{G}$} if any two
distinct greatest-type divisors of $x$ in $S$ satisfy the condition $\mathcal{G}$.
Moreover, we say that a set $S$ of positive integers {\it satisfies the
condition $\mathcal{G}$} if any element $x$ in $S$ satisfies that either
$|G_S(x)|\le 1$, or $|G_S(x)|\ge 2$ and $x$ satisfies the condition $\mathcal{G}$.
In \cite{[H-BAMS25]}, Hong showed that any FC set is a gcd-closed
set satisfying the condition $\mathcal{G}$. Furthermore, Hong raised the
following conjecture in Section 3 of \cite{[H-BAMS25]}.

\begin{conjecture}{\rm {\cite[Conjecture 3.4]{[H-BAMS25]}}} \label{p1.1}
Let $a$ and $b$ be positive integers with $a|b$ and let $S$ be
a gcd-closed set satisfying the condition $\mathcal G$. Then
$[S^a]\mid[S^b]$ in the ring $M_{|S|}(\mathbb Z)$.
\end{conjecture}

For the case $\max_{x\in S}\{|G_S (x)|\}=1$, by the theorem of Zhu and Li
\cite{[ZL-BAMS2023]} we know that this conjecture is true. For the case
$\max_{x\in S}\{|G_S (x)|\}=2$, by \cite{[WZ]} one knows that Conjecture
\ref{p1.1} holds. Furthermore, Zhu, Luo and Wan \cite{[ZLW]} confirmed
Conjecture \ref{p1.1} for the case $\max_{x\in S}\{|G_S (x)|\}=3$. But
Conjecture \ref{p1.1} is still kept open when $S$ is the general gcd-closed
set satisfying the condition $\mathcal{G}$.

In this paper, our main goal is to study the divisibility among power
LCM matrices. We introduce a new method to investigate Conjecture \ref{p1.1}.
The main result of this paper can be stated as follows.

\begin{theorem}\label{theorem 1.3}
Let $S$ be a gcd-closed set satisfying the condition
$\mathcal{G}$  and let $a$ and $b$ be positive integers
such that $a|b$. Then the $a$th power LCM matrix $[S^a]$
divides the $b$th power LCM matrix $[S^b]$
in the ring $M_{|S|}(\mathbb Z)$.
\end{theorem}

This paper is organized as follows. After introducing some necessary preliminary results,
we proceed to the core of our argument. In fact, we first recall and show in Section 2 some
results related to the arithmetic properties of gcd-closed sets that satisfy the condition
$\mathcal G$ and then present a formula for the inverse of power LCM matrices. Finally,
in Section 3, by proving a key lemma, we present the proof of Theorem \ref{theorem 1.3}.

The divisibility of matrices under consideration is invariant under relabeling the elements
of $S$. Specifically, for any permutation $\sigma$ on $\langle n\rangle$, we have
$[S^a]\mid [S^b]$ if and only if $[S^a_{\sigma}]\mid [S^b_{\sigma}]$. Therefore, without
loss of generality, we may assume throughout the rest of this paper that the set
$S=\{x_1,\cdots,x_n\}$ is ordered such that $x_1<\dots<x_n$.

\section{Preliminary lemmas}

Throughout this section, let $x_m \in S$ with $G_S(x_m)=\{x_{m_1}, \cdots , x_{m_s}\}$,
$2\le m\le n=|S|$ and $1\le s\le m-1$. Let $1\le k\le s$ and for any integers
$i_1,\cdots, i_k$ with $1 \le i_1<\cdots <i_k \le s$, we define
$$x_{m_{i_1\cdots i_k}}:=(x_{m_{i_1}},\cdots, x_{m_{i_k}}).$$
Then $x_{m_{1\cdots s}}=(x_{m_1}, \cdots, x_{m_s})$. In what follows, we recall
and show several results that are needed in the proof of our main result. For any
$j\in\langle s\rangle$, define $P_j:=\frac{x_m}{x_{m_j}}$. We begin with a known
result which is due to Zhu and Yu \cite{[ZY-PAMS2026]} and is about the arithmetic
property of the gcd-closed set $S$ which satisfies the condition $\mathcal{G}$.
\begin{lemma}\label{lemma 2.1} {\rm\cite{[ZY-PAMS2026]}}
Let $S$ be a gcd-closed set satisfying the condition
$\mathcal{G}$ and $x_m \in S$ with $G_S(x_m)=\{x_{m_1}, \cdots , x_{m_s}\}$. Let $k$ be an integer
with $1 \le k \le s$. Then for any integers
$i_1,\cdots,i_k$ with $1\le i_1<\cdots<i_k\le s$, we have
\begin{align*}
x_{m_{i_1\cdots i_k}} P_{i_1}\cdots P_{i_k}=x_m.
\end{align*}	
\end{lemma}

\begin{lemma}{\rm\cite{[Z-PMH2026]}}\label{lemma 2.2}
Let $S$ be a gcd-closed set satisfying the condition $\mathcal{G}$
and $x_l, x_m\in S$ with $G_S(x_m)=\{x_{m_1},\cdots,x_{m_s}\}$.
Let $1\le k\le s-1$. Assume that $(x_l,x_m)\mid x_{m_{1\cdots k}}$
and $(x_l,x_m) \nmid x_{m_j}$ for all integers $j$ with $k+1 \le j \le s$.
Let $1 \le i_1<\cdots < i_h \le k$ with $1\le h\le k$
and $k+1 \le j_1<\cdots<j_t \le s$ with $1\le t\le s-k$.
Then each of the following is true.	
		
{\rm (i).} $(x_l, x_{m_{i_1\cdots i_h}})=(x_l, x_m).$

{\rm (ii).} $[x_l,x_{m_{j_1\cdots j_t}}]=[x_l,x_m].$

{\rm (iii).}
$[x_{l}, x_{m_{i_1\cdots i_hj_1\cdots j_t}}]=[x_l,x_{m_{i_1\cdots i_h}}].$	
\end{lemma}

\begin{lemma} {\rm \cite[Theorem 3]{[BL-LMA1993]}} \label{lemma 2.3}
If $S$ is gcd-closed set of $n$ positive integers and $(f(S))$ is
nonsingular, then
$$((f(S))^{-1})_{ij}:={\underset{x_i|x_k\atop x_j|x_k}
{\sum}}\frac{c_{ik}c_{jk}}{\alpha_f(x_k)}$$
with
\begin{align}\label{eq2.1}
\alpha_f(x_k):={\underset{d|x_k\atop d\nmid x_t, x_t<x_k}{\sum}}(f*\mu)(d)
\end{align}
and
\begin{align}\label{eq2.2}
c_{ij}:=\sum _{dx_i|x_j\atop dx_i\nmid x_t, x_t<x_j}\mu(d).
\end{align}
\end{lemma}

\begin{lemma}{\rm\cite{[H-JA2004]}}\label{lemma 2.4}
Let $S$ be a gcd-closed set and let $\alpha_{f}(x_k)$
be given as in \eqref{eq2.1}. Then
$$
\alpha_{f}(x_k)=\sum\limits_{J\subseteq G_S(x_k)}(-1)^{|J|}
f\big(\gcd (J\cup \{x_k\})\big).
$$
\end{lemma}
		
\begin{lemma}{\rm\cite{[BL-LMA1993]}}\label{lemma 2.5}
If $S$ is gcd closed, then
\begin{align*}
\det(f(S))=\prod\limits_{k=1}^n\alpha_{f}(x_k)
\end{align*}
with $\alpha_{f}(x_k)$
being given as in \eqref{eq2.1}.
\end{lemma}
\begin{lemma}{\rm\cite{[ZY-PAMS2026]}}\label{lemma 2.6}
Let $S$ be a gcd-closed set satisfying the condition
$\mathcal G$ and let $x_m\in S$ be such that
$\{x_{m_1},\cdots,x_{m_s}\}\subseteq G_S(x_m)$.
For any real number $e$, define
$$
\beta_{\xi_e}(s,x_m)
:=x_m^e+\sum\limits_{t=1}^s(-1)^t\sum
\limits_{1\le i_1<\cdots<i_t\le s}x^e_{m_{i_1\cdots i_t}}.
$$
Then
\begin{align*}
\beta_{\xi_e}(s,x_m)=x^e_{m_{1\cdots s}}
\prod_{j=1}^s\big(\Big(\frac{x_m}{x_{m_j}}\Big)^e-1\big).
\end{align*}
\end{lemma}

For any real number $e$, the reciprocal arithmetic function $\frac{1}{\xi_e}$ is
defined for any positive integer $x$ by $\frac{1}{\xi_e}(x):=\frac{1}{x^e}$.
In fact, one has $\frac{1}{\xi_e}=\xi_{-e}$.

\begin{lemma}\label{lemma 2.7}
Let $S$ be a gcd-closed set satisfying the condition
$\mathcal G$ and $x_m\in S$ with
$G_S(x_m)=\{x_{m_1},\cdots,x_{m_s}\}$.
Then
$$
\alpha_{\frac{1}{\xi_a}}(x_m)
=x_{m_{1\cdots s}}^{-a} \prod_{j=1}^s\Big(\Big(\frac{x_m}{x_{m_j}}\Big)^{-a}-1\Big).
$$
\end{lemma}
\begin{proof}
Setting $f=\frac{1}{\xi_a}$ in Lemma \ref{lemma 2.4}, by Lemmas
\ref{lemma 2.4} and \ref{lemma 2.6}, we get that
\begin{align*}
\alpha_{\frac{1}{\xi_a}}(x_m)=\beta_{\xi_{-a}}(s,x_m)
=x^{-a}_{m_{1\cdots s}}\prod_{j=1}^s\Big(\Big(\frac{x_m}{x_{m_j}}\Big)^{-a}-1\Big)
\end{align*}
as required. So Lemma \ref{lemma 2.7} is proved.
\end{proof}

For any real number $e$ and for any gcd-closed set $S$ satisfying the condition $\mathcal{G}$,
Zhu and Yu \cite{[ZY-PAMS2026]} proved that the $e$th power LCM matrix $[S^e]$ is nonsingular
when $e\ne 0$ which confirmed another conjecture of Hong raised in \cite{[H-BAMS25]}. In the
following, we present the explicit formulas for the inverses of the $e$th power GCD matrix
$(S^e)$ and the $e$th power LCM matrix $[S^e]$.

\begin{lemma}\label{lemma 2.8}
Let $S$ be a gcd-closed set of $n$ positive integers and let $S$ satisfy the condition
$\mathcal G$. Then for any nonzero real number $e$, the $e$th power GCD matrix $(S^e)$
and the $e$th power LCM matrix $[S^e]$ are both nonsingular, and for any integers $i$
and $j$ with $1\le i,j\le n$, we have
$$
((S^e)^{-1})_{ij}:={\underset{x_i|x_k\atop x_j|x_k}{\sum}}
\frac{c_{ik}c_{jk}}{\alpha_{\xi_e}(x_k)}
$$
and
$$
([S^e]^{-1})_{ij}:=\frac{1}{x_i^e x_j^e}{\underset{x_i|x_k\atop x_j|x_k}{\sum}}
\frac{c_{ik}c_{jk}}{\alpha_{\frac{1}{\xi_e}}(x_k)},
$$
where $c_{ik}$ and $c_{jk}$ are defined as in {\rm (\ref{eq2.2})}.
\end{lemma}
\begin{proof}
Since
$${[x_i,x_j]}^e=\frac{x_i^e x_j^e}{{(x_i,x_j)}^e},$$
we have
\begin{align}\label{eq2.4}
[S^e]=&{\rm diag}(x_1^e,\cdots,x_n^e)\cdot
\Big(\frac{1}{\xi_e}(x_i,x_j)\Big)\cdot{\rm diag}(x_1^e,\cdots,x_n^e)\notag\\
=&:{\rm diag}(x_1^e,\cdots,x_n^e)\cdot B\cdot {\rm diag}(x_1^e,\cdots,x_n^e),
\end{align}
where $B:=\big(\frac{1}{\xi_e}(x_i,x_j)\big)=(\xi_{-e}(S))$ is the
$(-e)$th power GCD matrix defined on $S$. We take the determinants
on both sides of \eqref{eq2.4} and obtain that
\begin{align}\label{eq2.5}
\det[S^e]=\det B\cdot\prod\limits_{k=1}^nx_k^{2e}.
\end{align}
Since the two diagonal matrices in \eqref{eq2.4} are nonsingular, the
nonsingularity of the $e$th power LCM matrix $[S^e]$ is equivalent
to that of $B$. Thus we first prove that $B$ is nonsingular.

By Lemma \ref{lemma 2.5}, the determinant of $B$ is equal to
\begin{align}\label{I}
\det B=\prod\limits_{k=1}^n\alpha_{\frac{1}{\xi_e}}(x_k).
\end{align}
Hence it suffices to show that each factor $\alpha_{\frac{1}{\xi_e}}(x_k)$ is nonzero.
First, from Lemma \ref{lemma 2.4} we have
\begin{align}\label{II}
\alpha_{\frac{1}{\xi_e}}(x_1)=x_1^{-e}\ne0.
\end{align}
Next, for any $2\le m\le n=|S|$, let $G_S(x_m)=\{x_{m_1},\cdots, x_{m_s}\}$ with $s=|G_S(x_m)|\ge 1$.
For each $x_{m_j}\in G_S(x_m)$ ($1\le j\le s$), since $-e\ne 0$ and $x_m\ne x_{m_j}$, we have
$$
\Big(\frac{x_m}{x_{m_j}}\Big)^{-e}-1\neq0.
$$
Applying Lemma \ref{lemma 2.7} then yields that
\begin{align}\label{III}
\alpha_{\frac{1}{\xi_e}}(x_m)
=x_{m_{1\cdots s}}^{-e} \prod_{j=1}^s\Big(\Big(\frac{x_m}{x_{m_j}}\Big)^{-e}-1\Big)\ne 0.
\end{align}
Combining \eqref{I} to \eqref{III}, we conclude that $\det B\ne 0$, so $B$ is nonsingular.
Replacing $-e$ by $e$, this gives that the $e$th power GCD matrix $(S^e)$ is nonsingular
since $e\ne 0$. Moreover, together with \eqref{eq2.5}, this implies that $[S^e]$ is nonsingular.

Furthermore, we apply Lemma \ref{lemma 2.3} to the function $f=\xi_e$ and get the formula
for $(S^e)^{-1}$. Letting $f=\frac{1}{\xi_e}$ in Lemma \ref{lemma 2.3} gives the entrywise
formula for $B^{-1}$ as follows:
\begin{align}\label{eq2.6}
\big(B^{-1}\big)_{ij}
={\underset{x_i|x_k\atop x_j|x_k}
{\sum}}\frac{c_{i k}c_{j k}}{\alpha_{\frac{1}{\xi_e}}(x_k)}.
\end{align}
Finally, combining \eqref{eq2.4} and \eqref{eq2.6} yields the desired result immediately.

This completes the proof of Lemma \ref{lemma 2.8}.
\end{proof}

\begin{lemma}{\rm\cite{[Z-PMH2026]}}\label{lemma 2.9}
Let $S$ be a gcd-closed set satisfying the condition $\mathcal G$ with $x_m \in S$.
For any integer $r$ with $1\le r\le m$, if $c_{rm}$ is defined as in
{\rm (\ref{eq2.2})} and $G_S(x_m)=\{x_{m_1},\cdots, x_{m_s}\}$, then
\begin{align*}
c_{rm}=\left\{\begin{array}{cl}
1, &\hbox{if}\ r=m,\\
(-1)^k, &\hbox{if}\ r=m_{i_1i_2\cdots i_k}\ (1 \leq i_1 < \cdots < i_k \leq s,\ 1 \leq k \leq s),\\
0, &\hbox{otherwise}.
\end{array}\right.
\end{align*}
\end{lemma}

\section{A key lemma and proof of Theorem \ref{theorem 1.3}}
In this section, with the aid of the lemmas presented in Section 2,
we prove Theorem \ref{theorem 1.3}.
For any integers $l$ and $m$ with $1\le l,m\le |S|$,
we define the two-variable function $H$ as follows:
\begin{align}\label{eq 3.1}
H(l,m):=\frac{1}{\alpha_{\frac{1}{\xi_a}}(x_m)}\sum\limits_{x_r|x_m}\frac{c_{rm}[x_l, x_r]^b}{x_r^a}.
\end{align}
		
First of all, we prove a key lemma that plays a crucial role in the proof of Theorem \ref{theorem 1.3}.

\begin{lemma}\label{lemma 4.2}
Let $S$ be a gcd-closed set satisfying the condition $\mathcal{G}$
with $x_l,x_m\in S$, and let $G_S(x_m)=\{x_{m_1},\cdots, x_{m_s}\}$
with $s\ge 1$. Let $a$ and $b$ be positive integers. Then
\begin{align*}
H(l,m)=\left\{
\begin{aligned}
& (-1)^s \Big(\frac{x_l}{(x_l,x_m)}\Big)^b x_m^a x_{m_{1\cdots s}}^{b-a}
\prod_{j=1}^s\frac{\big(\frac{x_m}{x_{m_j}}\big)^{b-a}-1}{\big(\frac{x_m}{x_{m_j}}\big)^a-1}, \ \ \
\hbox{if}\ (x_l,x_m)\mid x_{m_{1\cdots s}},\\
& x_l^b, \ \ \  \hbox{if}\ x_m\mid x_l,\\
& (-1)^k \Big(\frac{x_l}{(x_l,x_m)}\Big)^bx_m^ax_{m_{1\cdots k}}^{b-a}
\prod_{j=1}^k\frac{\big(\frac{x_m}{x_{m_j}}\big)^{b-a}-1}{\big(\frac{x_m}{x_{m_j}}\big)^a-1},\ \ \
\hbox{if}\ (x_l,x_m)\mid x_{m_{1\cdots k}}\\
&\hbox{for some}\ 1\le k\le s-1\ \hbox{and}\ (x_l,x_m) \nmid x_{m_j}\ \hbox{for all}\ k+1 \le j \le s.
\end{aligned}
\right.
\end{align*}
\end{lemma}
		
\begin{proof}
By \eqref{eq 3.1} and Lemma \ref{lemma 2.9}, we have
\begin{align}\label{eq 3.2}
\alpha_{\frac{1}{\xi_a}}(x_m)H(l,m)=& \sum\limits_{x_r|x_m}\frac{c_{rm}[x_l, x_r]^b}{x_r^a}\notag\\
=&\frac{[x_l,x_m]^b}{x_m^a}+\sum\limits_{t=1}^s(-1)^t\sum\limits_{1\le j_{1}<\cdots<j_{t}\le s}
\frac{[x_l,x_{m_{j_{1}\cdots j_{t}}}]^b}{x_{m_{j_{1}\cdots j_{t}}}^a}.
\end{align}
	
Since $(x_l,x_m)|x_m$, it follows that either $(x_l,x_m)|x_{m_{1\cdots s}}$, or $(x_l,x_m)=x_m $,
or there exist an integer $k$ with $1\le k\le s-1$ and
$k$ integers $i_1,\cdots, i_k$ with $1\le i_1<\cdots <i_k\le s$
such that $(x_l,x_m)|x_{m_{i_1\cdots i_k}}$
and $(x_l,x_m) \nmid x_{m_j} $ for all $j \in \langle s\rangle\setminus \{ i_{1 },\cdots, i_{k}\}$.
So we only need to consider the following three cases.
			
{\sc Case 1.} $(x_l,x_m)|x_{m_{1\cdots s}}$. For $1 \le j_1<\cdots <j_t\le s$ with $1\le t\le s$,
since $x_{m_{1\cdots s}}| x_{m_{j_1\cdots j_t}}|x_m$, we have $(x_l,x_m)|x_{m_{j_1\cdots j_t}}|x_m$
and
\begin{align*}
(x_l,x_{m_{j_1\cdots j_t}}) = (x_l,(x_{m_{j_1\cdots j_t}}, x_m))
=((x_l,x_m),x_{m_{j_1\cdots j_t}})
=(x_l,x_m).
\end{align*}
Thus by \eqref{eq 3.2} and Lemma \ref{lemma 2.6}, we have
\begin{align*}
\alpha_{\frac{1}{\xi_a}}(x_m)H(l,m)=&\dfrac{x_l^bx_m^b}{x_m^a(x_l,x_m)^b}+\sum\limits_{t=1}^s(-1)^t
\sum\limits_{1\le j_{1}<\cdots<j_{t}\le s}\dfrac{x_l^bx^b_{m_{j_1\cdots j_t}}}{x^a_{m_{j_1\cdots j_t}}(x_l,x_{m_{j_1\cdots j_t}})^b}\notag\\
=&\dfrac{x_l^bx_m^{b-a}}{(x_l,x_m)^b} +\sum\limits_{t=1}^s(-1)^t
\sum\limits_{1\le j_{1}<\cdots<j_{t}\le s}\dfrac{x_l^bx^b_{m_{j_1\cdots j_t}}}{x^a_{m_{j_1\cdots j_t}}(x_l,x_m)^b}\notag\\
=&\dfrac{x_l^bx_m^{b-a}}{(x_l,x_m)^b}+\dfrac{x_l^b}{(x_l,x_m)^b}\sum\limits_{t=1}^s(-1)^t
\sum\limits_{1\le j_{1}<\cdots<j_{t}\le s}\frac{x^b_{m_{j_1\cdots j_t}}}{x^a_{m_{j_1\cdots j_t}}}\notag\\
=&\dfrac{x_l^b}{(x_l,x_m)^b}\Big(x_m^{b-a}+\sum\limits_{t=1}^s(-1)^t
\sum\limits_{1\le j_{1}<\cdots<j_{t}\le s}x_{m_{j_1\cdots j_t}}^{b-a}\Big)\notag\\
=&\dfrac{x_l^b}{(x_l,x_m)^b}\beta_{\xi_{b-a}}(s,x_m)\notag\\
=&\dfrac{x_l^b}{(x_l,x_m)^b}x_{m_{1\cdots s}}^{b-a}\prod_{j=1}^s\big(\Big(\frac{x_m}{x_{m_j}}\Big)^{b-a}-1\big).\notag
\end{align*}
Then it follows from Lemmas \ref{lemma 2.1} and \ref{lemma 2.7} that
\begin{align*}
H(l,m)=&\Big(\frac{x_l}{(x_l,x_m)}\Big)^b\frac{x_{m_{1\cdots s}}^{b-a}\prod_{j=1}^s\big(\big(\frac{x_m}{x_{m_j}}\big)^{b-a}-1\big)}
{x_{m_{1\cdots s}}^{-a}\prod_{j=1}^s\big(\big(\frac{x_m}{x_{m_j}}\big)^{-a}-1\big)}\\
=&\Big(\frac{x_l}{(x_l,x_m)}\Big)^bx_{m_{1\cdots s}}^b(-1)^s\prod_{j=1}^s\Big(\frac{x_m}{x_{m_j}}\Big)^a\cdot \prod_{j=1}^s
\frac{\big(\frac{x_m}{x_{m_j}}\big)^{b-a}-1}{\big(\frac{x_m}{x_{m_j}}\big)^a-1}\\
=&\Big(\frac{x_l}{(x_l,x_m)}\Big)^bx_{m_{1\cdots s}}^b(-1)^s\Big(\frac{x_m}{x_{m_{1\cdots s}}}\Big)^a
\prod_{j=1}^s \frac{\big(\frac{x_m}{x_{m_j}}\big)^{b-a}-1}{\big(\frac{x_m}{x_{m_j}}\big)^a-1}\\
=&\Big(\frac{x_l}{(x_l,x_m)}\Big)^bx_{m_{1\cdots s}}^{b-a}x_m^a (-1)^s
\prod_{j=1}^s \frac{\big(\frac{x_m}{x_{m_j}}\big)^{b-a}-1}{\big(\frac{x_m}{x_{m_j}}\big)^a-1}\\
=&(-1)^s \Big(\frac{x_l}{(x_l,x_m)}\Big)^b x_m^a x_{m_{1\cdots s}}^{b-a}
\prod_{j=1}^s \frac{\big(\frac{x_m}{x_{m_j}}\big)^{b-a}-1}{\big(\frac{x_m}{x_{m_j}}\big)^a-1}
\end{align*}
as expected.
		
{\sc Case 2.} $(x_l,x_m)=x_m$. Then $x_m|x_l$. Clearly,
$x_{m_{j_1\cdots j_t}}|x_m|x_l$ for $1 \le j_1<\cdots <j_t\le s$. This implies that
$$[x_l,x_m]=[x_l, x_{m_{j_1j_2\cdots j_t}}]=x_l.$$
Therefore, by \eqref{eq 3.2}, Lemmas \ref{lemma 2.6} and \ref{lemma 2.7}, we have
\begin{align*}
\alpha_{\frac{1}{\xi_a}}(x_m)H(l,m)=&\frac{x_l^b}{x_m^a}+\sum\limits_{t=1}^s(-1)^t\sum\limits_{1\le j_{1}<\cdots<j_{t}\le s}
\frac{[x_l, x_{m_{j_1\cdots j_t}}]^b}{x_{m_{j_1\cdots j_t}}^a}\\
=&\frac{x_l^b}{x_m^a}+\sum\limits_{t=1}^s(-1)^t\sum\limits_{1\le j_{1}<\cdots<j_{t}\le s}\frac{x_l^b}{x_{m_{j_1\cdots j_t}}^a}\\
=&x_l^b\Big(x_m^{-a}+\sum\limits_{t=1}^s(-1)^t\sum\limits_{1\le j_{1}<\cdots<j_{t}\le s}x_{m_{j_1\cdots j_t}}^{-a}\Big)\\
=&x_l^b\beta_{\xi_{-a}}(s,x_m)\\
=& x_l^b\alpha_{\frac{1}{\xi_a}}(x_m),
\end{align*}
and so $H(l,m)=x_l^b$ as desired.
			
{\sc Case 3.} There is an integer $k$ with $1\le k\le s-1$ and $k$ integers $i_1,\cdots, i_k$
with $1\le i_1<\cdots <i_k\le s-1$ such that $(x_l,x_m)|x_{m_{i_1\cdots i_k}}$
and $(x_l,x_m) \nmid x_{m_j} $ for all $j \in \langle s\rangle \setminus\{i_{1}, \cdots, i_{k}\}$.
WLOG, one may let $(x_l,x_m)\mid x_{m_{1\cdots k}}$ and $(x_l,x_m) \nmid x_{m_j} $ for all $k+1 \le j \le s$.
By \eqref{eq 3.2}, we get that
\begin{align}
\alpha_{\frac{1}{\xi_a}}(x_m)H(l,m)=&\frac{[x_l,x_m]^b}{x_m^a}+\sum_{h=1}^{k}(-1)^{h}\sum\limits_{1 \le i_1<\cdots<i_h\le k}
\frac{[x_l,x_{m_{i_1\cdots i_h}}]^b}{x_{m_{i_1\cdots i_h}}^a}\notag\\
&+\sum_{t=1}^{s-k}(-1)^{t}\sum\limits_{k+1 \le j_1<\cdots <j_t \le s}\frac{[x_l,x_{m_{j_1\cdots j_t}}]^b}{x_{m_{j_1\cdots j_t}}^a}\notag\\
&+\sum_{h=1}^{k}\sum_{t=1}^{s-k}(-1)^{h+t}\sum\limits_{1 \le i_1<\cdots<i_h\le k \atop k+1 \le j_1<\cdots <j_t \le s}
\frac{[x_l,x_{m_{i_1\cdots i_hj_1\cdots j_t}}]^b}{x_{m_{i_1\cdots i_hj_1\cdots j_t}}^a}\notag\\
=&E_1+E_2,\notag
\end{align}
where
\begin{equation*}
E_1:=\frac{[x_l,x_m]^b}{x_m^a}
+\sum_{h=1}^{k}(-1)^{h}\sum\limits_{1 \le i_1<\cdots<i_h\le k}\frac{[x_l,x_{m_{i_1\cdots i_h}}]^b}{x_{m_{i_1\cdots i_h}}^a}
\end{equation*}
and
\begin{align*}
E_2:=&\sum_{t=1}^{s-k}(-1)^{t}\sum\limits_{k+1 \le j_1<\cdots <j_t \le s}\frac{[x_l,x_{m_{j_1\cdots j_t}}]^b}{x_{m_{j_1\cdots j_t}}^a}\nonumber\\
&+\sum_{h=1}^{k}\sum_{t=1}^{s-k}(-1)^{h+t}\sum\limits_{1 \le i_1<\cdots<i_h\le k \atop k+1 \le j_1<\cdots <j_t \le s}
\frac{[x_l,x_{m_{i_1\cdots i_hj_1\cdots j_t}}]^b}{x_{m_{i_1\cdots i_hj_1\cdots j_t}}^a}.
\end{align*}
			
First of all, we compute $E_1$. By Lemma \ref{lemma 2.2} (i) and Lemma \ref{lemma 2.6}, we obtain that
\begin{align}
E_1&=\frac{[x_l,x_m]^b}{x^a_m}+\sum_{h=1}^{k}(-1)^{h}\sum\limits_{1 \le i_1<\cdots<i_h\le k}
\frac{[x_l,x_{m_{i_1\cdots i_h}}]^b}{x_{m_{i_1\cdots i_h}}^a}\nonumber\\
&=\frac{x_l^bx_m^b}{x^a_m(x_l,x_m)^b}+\sum_{h=1}^{k}(-1)^{h}\sum\limits_{1 \le i_1<\cdots<i_h\le k}
\frac{x_l^bx_{m_{i_1\cdots i_h}}^b}{x_{m_{i_1\cdots i_h}}^a(x_l,x_{m_{i_1\cdots i_h}})^b}\nonumber\\
&=\frac{x_l^bx_m^{b-a}}{(x_l,x_m)^b}+\sum_{h=1}^{k}(-1)^{h}\sum\limits_{1\le i_1<\cdots<i_h\le k}
\frac{x_l^bx_{m_{i_1\cdots i_h}}^{b-a}}{(x_l,x_m)^b}\nonumber\\
&=\Big(\frac{x_l}{(x_l,x_m)}\Big)^b\Big(x_m^{b-a}+\sum_{h=1}^{k}(-1)^{h}
\sum\limits_{1\le i_1<\cdots<i_h\le k}x_{m_{i_1\cdots i_h}}^{b-a}\Big)\nonumber\\
&=\Big(\frac{x_l}{(x_l,x_m)}\Big)^b\beta_{\xi_{b-a}}(k,x_m)\notag\\
&=\Big(\frac{x_l}{(x_l,x_m)}\Big)^bx_{m_{1\cdots k}}^{b-a}\prod_{j=1}^k\Big(\Big(\frac{x_m}{x_{m_j}}\Big)^{b-a}-1\Big).\notag
\end{align}

Subsequently, we calculate $E_2$. From Lemmas \ref{lemma 2.1}, \ref{lemma 2.2} (i) to (iii),
and \ref{lemma 2.6}, we can deduce that
\begin{align}
E_2=&\sum_{t=1}^{s-k}(-1)^{t}\sum\limits_{k+1 \le j_1<\cdots <j_t \le s}\frac{[x_l,x_{m_{j_1\cdots j_t}}]^b}{x_{m_{j_1\cdots j_t}}^a}\notag\\
&+\sum_{h=1}^{k}\sum_{t=1}^{s-k}(-1)^{h+t}\sum\limits_{1 \le i_1<\cdots<i_h\le k \atop k+1 \le j_1<\cdots <j_t\le s}
\frac{[x_l,x_{m_{i_1\cdots i_hj_1\cdots j_t}}]^b}{x_{m_{i_1\cdots i_hj_1\cdots j_t}}^a}\notag\\
=&\sum_{t=1}^{s-k}(-1)^{t}\sum\limits_{k+1\le j_1<\cdots <j_t \le s}\frac{[x_l,x_m]^b}{x_{m_{j_1\cdots j_t}}^a}\notag\\
&+\sum_{h=1}^{k}\sum_{t=1}^{s-k}(-1)^{h+t}\sum\limits_{1\le i_1<\cdots<i_h\le k \atop k+1\le j_1<\cdots <j_t\le s}
\frac{[x_l,x_{m_{i_1\cdots i_h}}]^b}{x_{m_{i_1\cdots i_hj_1\cdots j_t}}^a}\notag\\
=&\sum_{t=1}^{s-k}(-1)^{t}\sum\limits_{k+1\le j_1<\cdots <j_t\le s}\frac{x_l^bx_m^b}{(x_l,x_m)^bx_{m_{j_1\cdots j_t}}^a}\notag\\
&+\sum_{h=1}^{k}\sum_{t=1}^{s-k}(-1)^{h+t}\sum\limits_{1\le i_1<\cdots<i_h\le k \atop k+1\le j_1<\cdots <j_t\le s}
\frac{x_l^bx_{m_{i_1\cdots i_h}}^b}{(x_l,x_{m_{i_1\cdots i_h}})^bx_{m_{i_1\cdots i_hj_1\cdots j_t}}^a}\notag\\
=&\sum_{t=1}^{s-k}(-1)^{t}\sum\limits_{k+1\le j_1<\cdots <j_t\le s}\frac{x_l^bx_m^b}{(x_l,x_m)^bx_{m_{j_1\cdots j_t}}^a}\notag\\
&+\sum_{h=1}^{k}\sum_{t=1}^{s-k}(-1)^{h+t}\sum\limits_{1\le i_1<\cdots<i_h\le k\atop k+1 \le j_1<\cdots <j_t\le s}
\frac{x_l^bx_{m_{i_1\cdots i_h}}^b}{(x_l,x_m)^bx_{m_{i_1\cdots i_hj_1\cdots j_t}}^a}\notag\\
=&\Big(\frac{x_l}{(x_l,x_m)}\Big)^b\Big(\sum_{t=1}^{s-k}(-1)^{t}\sum\limits_{k+1\le j_1<\cdots <j_t\le s}\frac{x_m^b}{\frac{x_m^a}{(P_{j_1}\cdots P_{j_t})^a}}\notag\\
&+\sum_{h=1}^{k}\sum_{t=1}^{s-k}(-1)^{h+t}\sum\limits_{1\le i_1<\cdots<i_h\le k \atop k+1\le j_1<\cdots <j_t \le s}
\frac{\frac{x_m^b}{(P_{i_1}\cdots P_{i_h})^b}}{\frac{x_m^a}{(P_{i_1}\cdots P_{i_h}P_{j_1}\cdots P_{j_t})^a}}\Big)\notag\\
=&\Big(\frac{x_l}{(x_l,x_m)}\Big)^b\Big(\sum_{t=1}^{s-k}(-1)^{t}\sum\limits_{k+1\le j_1<\cdots <j_t \le s}x_m^{b-a}(P_{j_1}\cdots P_{j_t})^a\notag\\
&+\sum_{h=1}^{k}\sum_{t=1}^{s-k}(-1)^{h+t}\sum\limits_{1\le i_1<\cdots<i_h\le k\atop k+1\le j_1<\cdots <j_t\le s}
\frac{x_m^{b-a}}{(P_{i_1}\cdots P_{i_h})^{b-a}}\cdot (P_{j_1}\cdots P_{j_t})^a\Big)\notag\\
=&\Big(\frac{x_l}{(x_l,x_m)}\Big)^b\Big(\sum_{t=1}^{s-k}(-1)^{t}\sum\limits_{k+1\le j_1<\cdots <j_t\le s}(P_{j_1}\cdots P_{j_t})^a\Big(x_m^{b-a}\notag\\
&+\sum_{h=1}^{k}(-1)^h\sum\limits_{1 \le i_1<\cdots<i_h\le k}\frac{x_m^{b-a}}{(P_{i_1}\cdots P_{i_h})^{b-a}}\Big)\Big)\notag\\
=&\Big(\frac{x_l}{(x_l,x_m)}\Big)^b\Big(\sum_{t=1}^{s-k}(-1)^{t}\sum\limits_{k+1 \le j_1<\cdots <j_t \le s}(P_{j_1}\cdots P_{j_t})^a\Big(x_m^{b-a}\notag\\
&+\sum_{h=1}^{k}(-1)^h\sum\limits_{1\le i_1<\cdots<i_h\le k}x_{m_{i_1\cdots i_h}}^{b-a}\Big)\Big)\notag\\
=&\Big(\frac{x_l}{(x_l,x_m)}\Big)^b\Big(\sum_{t=1}^{s-k}(-1)^{t}\sum\limits_{k+1\le j_1<\cdots <j_t \le s}(P_{j_1}\cdots P_{j_t})^a\beta_{\xi_{b-a}}(k,x_m)\Big)\notag\\
=&\Big(\frac{x_l}{(x_l,x_m)}\Big)^b\Big(\sum_{t=1}^{s-k}(-1)^{t}\sum\limits_{k+1\le j_1<\cdots <j_t \le s}(P_{j_1}\cdots P_{j_t})^ax_{m_{1\cdots k}}^{b-a}\notag\\
&\times \prod_{j=1}^k\Big(\Big(\frac{x_m}{x_{m_j}}\Big)^{b-a}-1\Big)\Big)\notag\\
=&\Big(\frac{x_l}{(x_l,x_m)}\Big)^bx_{m_{1\cdots k}}^{b-a}
\prod_{j=1}^k\Big(\Big(\frac{x_m}{x_{m_j}}\Big)^{b-a}-1\Big)\sum_{t=1}^{s-k}(-1)^{t}\sum\limits_{k+1\le j_1<\cdots <j_t \le s}(P_{j_1}\cdots P_{j_t})^a\notag\\
=&\Big(\frac{x_l}{(x_l,x_m)}\Big)^bx_{m_{1\cdots k}}^{b-a}
\prod_{j=1}^k\Big(\Big(\frac{x_m}{x_{m_j}}\Big)^{b-a}-1\Big)\sum_{t=1}^{s-k}(-1)^{t}\sum\limits_{k+1\le j_1<\cdots <j_t \le s}
\frac{x_m^a}{\frac{x_m^a}{(P_{j_1}\cdots P_{j_t})^a}}\notag\\
=&\Big(\frac{x_l}{(x_l,x_m)}\Big)^bx_{m_{1\cdots k}}^{b-a}
\prod_{j=1}^k\Big(\Big(\frac{x_m}{x_{m_j}}\Big)^{b-a}-1\Big)\sum_{t=1}^{s-k}(-1)^{t}\sum\limits_{k+1 \le j_1<\cdots <j_t \le s}\frac{x_m^a}{x_{m_{j_1\cdots j_t}}^a}\notag\\
=&\Big(\frac{x_l}{(x_l,x_m)}\Big)^bx_{m_{1\cdots k}}^{b-a}\prod_{j=1}^k\Big(\Big(\frac{x_m}{x_{m_j}}\Big)^{b-a}-1\Big)\notag\\
&\times\Big(x_m^a\Big(x_m^{-a}+\sum_{t=1}^{s-k}(-1)^{t}\sum^{s-k}\limits_{k+1 \le j_1<\cdots <j_t \le s}x_{m_{j_1\cdots j_t}}^{-a}\Big)-1\Big)\notag\\
=&\Big(\frac{x_l}{(x_l,x_m)}\Big)^bx_{m_{1\cdots k}}^{b-a}\prod_{j=1}^k\Big(\Big(\frac{x_m}{x_{m_j}}\Big)^{b-a}-1\Big)
(x_m^a\beta_{\xi_{-a}}(s-k,x_m)-1)\notag\\
=&\Big(\frac{x_l}{(x_l,x_m)}\Big)^bx_{m_{1\cdots k}}^{b-a}\prod_{j=1}^k\Big(\Big(\frac{x_m}{x_{m_j}}\Big)^{b-a}-1\Big)
\Big(x_m^a\Big(x_{m_{(k+1)\cdots s}}^{-a}\prod_{j=k+1}^s\Big(\Big(\frac{x_m}{x_{m_j}}\Big)^{-a}-1\Big)-1\Big).\notag
\end{align}
Thus
$$
E_1+E_2=\Big(\frac{x_l}{(x_l,x_m)}\Big)^bx_{m_{1\cdots k}}^{b-a}x_m^ax_{m_{(k+1)\cdots s}}^{-a}
\prod\limits_{j=1}^k\Big(\Big(\frac{x_m}{x_{m_j}}\Big)^{b-a}-1\Big)\prod\limits_{j=k+1}^s\Big(\Big(\frac{x_m}{x_{m_j}}\Big)^{-a}-1\Big),
$$
and so by Lemmas \ref{lemma 2.1} and \ref{lemma 2.7}, we have
\begin{align*}
H(l,m)=&\frac{E_1+E_2}{\alpha_{\frac{1}{\xi_a}}(x_m)}\\
=&\Big(\frac{x_l}{(x_l,x_m)}\Big)^bx_m^ax_{m_{1\cdots k}}^{b-a}x_{m_{(k+1)\cdots s}}^{-a}x_{m_{1\cdots s}}^a
\frac{\prod_{j=1}^k\big(\big(\frac{x_m}{x_{m_j}}\big)^{b-a}-1\big)}{\prod_{j=1}^k\big(\big(\frac{x_m}{x_{m_j}}\big)^{-a}-1\big)}\\
=&\Big(\frac{x_l}{(x_l,x_m)}\Big)^bx_m^ax_{m_{1\cdots k}}^{b-a}x_{m_{(k+1)\cdots s}}^{-a}x_{m_{1\cdots s}}^a\\
&\times(-1)^k\Big(\prod_{j=1}^k\frac{x_m}{x_{m_j}}\Big)^a\frac{\prod_{j=1}^k\big(\big(\frac{x_m}{x_{m_j}}\big)^{b-a}-1\big)}
{\prod_{j=1}^k\big(\big(\frac{x_m}{x_{m_j}}\big)^a-1\big)}\\
=&\Big(\frac{x_l}{(x_l,x_m)}\Big)^bx_m^ax_{m_{1\cdots k}}^{b-a}x_{m_{(k+1)\cdots s}}^{-a}x_{m_{1\cdots s}}^a\\
&\times(-1)^k(P_1\cdots P_k)^a\prod_{j=1}^k\frac{\big(\frac{x_m}{x_{m_j}}\big)^{b-a}-1}{\big(\frac{x_m}{x_{m_j}}\big)^a-1}\\
=&\Big(\frac{x_l}{(x_l,x_m)}\Big)^bx_m^ax_{m_{1\cdots k}}^{b-a}\frac{(P_{k+1}\cdots P_s)^a}{x_m^a}x_{m_{1\cdots s}}^a\\
&\times(-1)^k(P_1\cdots P_k)^a\prod_{j=1}^k\frac{\big(\frac{x_m}{x_{m_j}}\big)^{b-a}-1}{\big(\frac{x_m}{x_{m_j}}\big)^a-1}\\
=&\Big(\frac{x_l}{(x_l,x_m)}\Big)^bx_m^ax_{m_{1\cdots k}}^{b-a}(-1)^k
\prod_{j=1}^k\frac{\big(\frac{x_m}{x_{m_j}}\big)^{b-a}-1}{\big(\frac{x_m}{x_{m_j}}\big)^a-1}\\
=& (-1)^k \Big(\frac{x_l}{(x_l,x_m)}\Big)^bx_m^ax_{m_{1\cdots k}}^{b-a}
\prod_{j=1}^k\frac{\big(\frac{x_m}{x_{m_j}}\big)^{b-a}-1}{\big(\frac{x_m}{x_{m_j}}\big)^a-1}
\end{align*}
as asserted.

This finishes the proof of Lemma \ref{lemma 4.2}.
\end{proof}

Finally, we conclude this paper by giving the proof of Theorem \ref{theorem 1.3}.
		
\begin{proof}[Proof of Theorem \ref{theorem 1.3}]
Let $S$ be a gcd-closed set which satisfies the condition $\mathcal{G}$
and let $a$ and $b$ be positive integers with $a\mid b$.
For any integers $i$ and $j$ with $1 \le i,j \le |S|$, by Lemma \ref{lemma 2.8}, we have
\begin{align*}
{\left([S^b][S^a]^{-1}\right)}_{ij}
=&\sum_{r=1}^{|S|}[x_i, x_r]^b\dfrac{1}{x_r^ax_j^a}\sum_{x_r|x_t\atop x_j|x_t}
\dfrac{c_{rt}c_{jt}}{\alpha_{\frac{1}{\xi_a}}(x_t)}\\
=&\sum\limits_{x_j|x_t}c_{jt}\frac{1}{x_j^a\alpha_{\frac{1}{\xi_a}}(x_t)}\sum\limits_{x_r|x_t}
\frac{c_{rt}[x_i, x_r]^b}{x_r^a}\\
=&\sum\limits_{x_j|x_t}c_{jt}\frac{H(i,t)}{x_j^a}
\end{align*}
with $H(i,t)$ being defined as in \eqref{eq 3.1}.
		
In what follows we prove that
$${\left( [S^b][S^a]^{-1}\right)}_{ij}\in\mathbb Z$$
for all integers $i$ and $j$ with $1\le i,j \le |S|$.
By the definition of $c_{jt} $ as in \eqref{eq2.2}, we know that $ c_{jt} \in \mathbb{Z}$.
So we need only to show that for any positive integer $t$ with $x_j\mid x_t$, we have
$$\frac{H(i,t)}{x_j^a}\in \mathbb{Z}.$$
This will be shown in what follows.
			
For any given integers $i$ and $t$ with $1 \le i,t\le |S|$ and $x_j|x_t$, we consider the following two cases.
			
{\sc Case 1.} $t=1$. Then $x_t=x_1=\gcd(S)$. At this moment, we must have $x_j=x_1$.
By Lemmas \ref{lemma 2.4} and \ref{lemma 2.9}, we obtain that
\begin{align*}
\frac{H(i,t)}{x_j^a}=\frac{H(i,1)}{x_1^a}=\frac{1}{x_1^a\alpha_{\frac{1}{\xi_a}}(x_1)}\sum\limits_{x_r|x_1}\frac{c_{r1}
[x_i, x_r]^b}{x_r^a}=\frac{[x_i,x_1]^b}{x_1^a}=\frac{x_i^b}{x_1^a}=x_i^{b-a}\Big(\frac{x_i}{x_1}\Big)^a\in \mathbb Z
\end{align*}
as one expects.
				
{\sc Case 2.} $t\ge 2$. Since $S$ is gcd closed, one has $x_1\mid x_t$ and so $|G_S(x_t)|\ge 1$.
Write $G_S(x_t):=\{x_{t_1},\cdots, x_{t_s}\}$ with $s\ge 1$.
Since $a\mid b$ and $S$ satisfies the condition $\mathcal{G}$ as well as $x_j\mid x_t$, by Lemma \ref{lemma 4.2}, we have
\begin{align*}
\frac{H(i,t)}{x_j^a}=\left\{
\begin{aligned}
&(-1)^s\Big(\frac{x_i}{(x_i,x_t)}\Big)^bx_{t_{1\cdots s}}^{b-a}\Big(\frac{x_t}{x_j}\Big)^a
\prod_{j=1}^s\frac{\big(\frac{x_t}{x_{t_j}}\big)^{b-a}-1}{\big(\frac{x_t}{x_{t_j}}\big)^a-1}\in\mathbb Z,
\ \ \hbox{if}\ (x_i,x_t)\mid x_{t_{1\cdots s}},\\
&x_i^{b-a}\Big(\frac{x_i}{x_j}\Big)^a\in\mathbb Z, \ \ \hbox{if}\ x_t\mid x_i,\\
&(-1)^k\Big(\frac{x_i}{(x_i,x_t)}\Big)^b\Big(\frac{x_t}{x_j}\Big)^ax_{t_{1\cdots k}}^{b-a}
\prod_{j=1}^k\frac{\big(\frac{x_t}{x_{t_j}}\big)^{b-a}-1}{\big(\frac{x_t}{x_{t_j}}\big)^a-1}\in\mathbb Z,
\ \hbox{if}\ (x_i,x_t)\mid x_{t_{1\cdots k}}\\
&\hbox{for some}\ 1\le k\le s-1\ \hbox{and}\ (x_i,x_t) \nmid x_{t_j}\ \hbox{for all}\ k+1 \le j \le s.
\end{aligned}
\right.
\end{align*}
It then follows that
$$\frac{H(i,t)}{x_j^a}\in\mathbb Z$$
as required.

This concludes the proof of Theorem \ref{theorem 1.3}.
\end{proof}
		
\begin{center}
{\sc Acknowledgements}
\end{center}
The author is grateful to the anonymous referee for insightful and invaluable
suggestions and comments that improved greatly the presentation of the paper.\\

	\end{document}